\documentclass[12pt,twoside]{amsart}
\usepackage{latexsym,amsmath,amsopn,amssymb,amsthm,amsfonts}
\usepackage[T2A]{fontenc}
\usepackage[cp1251]{inputenc}
\usepackage[russian,english]{babel}
\usepackage{graphicx}

\newtheorem{theorem}{Theorem}

\newtheorem{proposition}{Proposition}

\newtheorem{remark}{Remark}

\newtheorem{definition}{Definition}

\newtheorem{corollary}{Corollary}

\newcommand{\Ad}{\operatorname{Ad}}

\newcommand{\Arccos}{\operatorname{arccos}}

\newcommand{\sgn}{\operatorname{sgn}}
\newcommand{\Sh}{\operatorname{sh}}
\newcommand{\Ch}{\operatorname{ch}}

\newcommand{\Arch}{\operatorname{arch}}

\newcommand{\Th}{\operatorname{th}}
\newcommand{\Mink}{\operatorname{Mink}}

\newcommand{\Exp}{\operatorname{Exp}}
\newcommand{\Ann}{\operatorname{Ann}}
\newcommand{\Sim}{\operatorname{Sim}}

\begin{document}

\begin{flushleft}
UDK 519.46 + 514.763 + 512.81 + 519.9 + 517.911
\end{flushleft}
\begin{flushleft}
MSC 22E30, 49J15, 53C17
\end{flushleft}

\title[group $SO_0(2,1)$]{Sub-Riemannian distance on the Lie group $SO_0(2,1)$}
\author{V.\,N.\,Berestovskii, I.\,A.\,Zubareva}
\thanks{The work is partially supported by the Russian Foundation for Basic Research (Grant 14-01-00068-a), a grant of the Government of the Russian Federation for the State Support of Scientific Research (Agreement 14.B25.31.0029), and the State Maintenance Program for the Leading Scientific Schools of the Russian Federation (Grant NSh-2263.2014.1)}
\address{V.N.Berestovskii}
\address{Sobolev Institute of Mathematics SD RAS,  \newline 4 Acad. Koptyug avenue, 630090, Novosibirsk, Russia}
\email{vberestov@inbox.ru}
\address{I.A.Zubareva}
\address{Sobolev Institute of Mathematics SD RAS, Omsk Branch,  \newline 13 Pevtsova street, 644043, Omsk, Russia}
\email{i\_gribanova@mail.ru}
\maketitle
\maketitle {\small
\begin{quote}
\noindent{\sc Abstract.}
A left-invariant sub-Riemannian metric $d$ on the shortened Lorentz group $SO_0(2,1)$ under the condition that $d$ is right-invariant relative to the orthogonal Lie subgroup $1\otimes SO(2)$ is studied. The distance between arbitrary two elements, the cut locus (as the union of the subgroup $1\otimes SO(2)$ with the antipodal set to the submanifold of symmetric matrices in the open solid torus $SO_0(2,1)$), and the conjugate set for the unit are found for $(SO_0(2,1),d).$
\end{quote}}

{\small
\begin{quote}
\end{quote}}

\section*{Introduction}

 In paper \cite{Ber} are found geodesics and shortest arcs for left-invariant and $SO(2)$--right-invariant sub-Riemannian metric $d$ on the Lie--Lorentz group (more precisely, so-called  the shortened Lorentz group) $SO_0(2,1),$ where $SO(2)\subset SO_0(2,1).$

In this paper, we find the cut locus (as the union of the subgroup $1\otimes SO(2)$ with the antipodal set to the submanifold of symmetric matrices in the open solid torus $SO_0(2,1)$) and the conjugate set for the unit, and we also compute distances between arbitrary elements in the metric space $(SO_0(2,1),d).$
We got analogous results for other groups in our papers  \cite{BerZub1}, \cite{BerZub2}. In this work, we apply the same geometric ideas and  interpretations of geodesics and shortest arcs as in \cite{Ber}, \cite{BerZ}.

One can give the following natural geometric description of the metric $d.$ The Lie group $SO_0(2,1)$ can be interpreted as an effective transitive group of all orientation-preserving isometries of the Lobachevskii plane $L^2$ with constant Gaussian curvature $-1$. By choosing any unit tangent vector $v_0$ on $L^2$ at  some point $w_0$, the equality $f(g)=dg(v_0),$ $g\in SO_0(2,1)$, determines a natural diffeomorphism $f: SO_0(2,1)\rightarrow L^2_1$ onto the space of unit tangent vectors on $L^2$. The space $L^2_1$ admits a natural Riemannian metric (scalar product) $g_1$ by Sasaki  (see \cite{Sas} or the tensor $g_1$ in the section $1K$ in Besse book \cite{Bes}) and an inner metric $\rho,$ associated with it. In addition, canonical projection $p: (L^2_1,g_1)\rightarrow L^2$ (or, which is equivalent, $p: SO_0(2,1)\rightarrow  SO_0(2,1)/SO(2)$)) is a \textit{Riemannian submersion} \cite{Bes}.
The metric $d$ is defined by left-invariant totally nonholonomic distribution $D$ on $SO_0(2,1),$ which is orthogonal to fibers of submersion $p,$ and the restriction of the scalar product $g_1$ to $D.$ The corresponding distribution $D$ on $L^2_1$ is nothing else than the restriction to $L^2_1$ of horizontal distribution of the Levi--Civita connection \cite{Bes} for $L^2.$

It follows from the above facts that hold the next statements.

1) The canonical projection
\begin{equation}
\label{subm}
p: (SO_0(2,1),d)\rightarrow  L^2
\end{equation}
is a \textit{submetry} \cite{BG}, a natural generalization of Riemannian submersion.

2) Identifying $SO_0(2,1)$ with $L^2_1$ by the map $f$, for any piece-wise horizontal (i.e. tangent to $D$) smooth path
$\gamma(t)$ , $0\leq t \leq t_1,$ in $SO_0(2,1)$, $f(\gamma(t))$ with $0\leq t \leq t_1$ is a parallel vector field
(in the \textit{Lobachevskii plane!}) along the projection $p(\gamma(t))$,
$0\leq t \leq t_1,$ in the sense of \cite{Pog}, with initial unit tangent vector $f(\gamma(0))=d\gamma(0)(v_0)\in L^2_1$, see \cite{BerZ}.

3) Additionally, if $p(\gamma(t))$, $0\leq t \leq t_1$, is a closed path with no self-intersection, bounding a region in
$L^2$ with the area $S\leq \pi,$ then $S=\rho(f(\gamma(0)),f(\gamma(t_1)))$; at the same time a preimage  $p^{-1}(x)$ of any point $x\in L^2$ is a circle with the length $2\pi$ in the metric $\rho$, $(p\circ f)^{-1}(x)=gSO(2)$ if $p(f(g))=x$, and the metric $\rho\circ (f\times f)$ is invariant relative to the right multiplication by elements of the subgroup $SO(2)\subset SO_0(2,1).$

4) The projection (\ref{subm}) maps geodesics in the space $(SO_0(2,1),d)$ equilongally to the curves of constant geodesic curvature, namely,  geodesics, equidistant curves, horocycles, or circles in $L^2$. Every line segment of such curve is a solution of corresponding isoperimetric Dido's problem in $L^2;$ in other words, in view of 3), is a solution of some isoholonomic problem.

\section{Preliminaries}

In this section, we will recall necessary notions and results from the paper \cite{Ber}.

The \textit{pseudoeuclidean space} $E^{n,1}$ or the \textit{Minkowski space-time} $\Mink^{n+1}$, where $n+1\geq 2,$  is the vector space $\mathbb{R}^{n+1}$ with \textit{pseudoscalar product} $\{(t,x),(s,y)\}:=-ts+(x,y).$ Here $(x,y)=xy^T$ is the \textit{standard scalar product of vectors} $x,y\in \mathbb{R}^n$, $T$ is the transposing operator. The \textit{Lorentz group} $SO_0(n,1)$ is the connected component of the unit in the group $P(n,1)$ of all linear \textit{pseudoisometric} (i.e., preserving the pseudoscalar product
$\{\cdot,\cdot\}$) transformations of the space-time $\Mink^{n+1}.$

In the canonical base of the  space $\mathbb{R}^{n+1},$ the elements of the group $P(n,1)$ are given by real
$(n+1)\times (n+1)$-matrixes $C$ with condition
\begin{equation}
\label{cps}
C^{-1}=IC^TI,
\end{equation}
where $I=-1\otimes E_n$ is the matrix of the \textit{time reversing operator} $(t,x)\rightarrow (-t,x)$.

\begin{remark}
\label{rem1}
The group $SO_0(n,1)$ consists of those elements in $P(n,1)$ which
simul-\,taneously preserve the time direction and orientation of the space $E^{n,1}$,  i.e. are given by matrices $C$ with conditions $c_{11}\geq 1$ and $\det(C)=1$. Moreover, $C\in SO(n):=1\otimes SO(n)$ if and only if $C\in SO_0(n,1)$ and $c_{11}=1.$ The matrix $C\in SO_0(n,1)$ commutes with $I$ if and only if $C\in SO(n).$
\end{remark}

The Lie algebra $\frak{so}(n,1)$ of the Lie groups $P(n,1)$ and $SO_0(n,1)$ is defined by the equality
\begin{equation}
\label{r1}
\frak{so}(n,1)=I\cdot\frak{so}(n+1),
\end{equation}
where $\frak{so}(n+1)$ is the Lie algebra of the Lie group $SO(n+1),$ consisting of all real skew-symmetric
$(n+1)\times (n+1)$-matrices.

We shall be interested in the case $n=2.$ In view of the equality (\ref{r1}), the matrices
\begin{equation}
\label{abc}
a=\left(\begin{array}{ccc}
0 & 1 & 0 \\
1 & 0 & 0 \\
0 & 0 & 0
\end{array}\right),\quad
b=\left(\begin{array}{ccc}
0 & 0 & 1 \\
0 & 0 & 0 \\
1 & 0 & 0
\end{array}\right),\quad
c=\left(\begin{array}{ccc}
0 & 0 & 0 \\
0 & 0 & -1 \\
0 & 1 & 0
\end{array}\right)
\end{equation}
constitute a basis of the Lie algebra $\frak{so}(2,1).$

Let $e$ be the unit of $SO_0(2,1)$, $\Delta(e)$ denotes the linear span of the vectors $a,\, b$.
Let us define the scalar product $\langle\cdot,\cdot\rangle$ on $D(e)$
with the orthonormal basis $a,\,b$. It follows from (\ref{abc}) that
\begin{equation}
\label{lbr}
[a,b]=-c,\quad [b,c]=a,\quad [c,a]=b.
\end{equation}
In consequence of relations (\ref{lbr}), the next statements hold.

1. The left-invariant disribution $D$ on the Lie group $SO_0(2,1)$ with given $D(e)$ is totally nonholonomic, and the pair $(D(e),\langle\cdot,\cdot\rangle)$ defines the left-invariant sub-Riemannian metric  $d$ on $SO_0(2,1).$

2. $D(e)$ and $\langle\cdot,\cdot\rangle$ are invariant relative to the group $\Ad(SO(2)),$ $SO(2)\subset SO_0(2,1).$

3. The metric $d$ is invariant relative to conjugation of the group $SO_0(2,1)$ by elements of the subgroup $SO(2).$

4. The metric $d$  is invariant relative to the right shifts of the group $SO_0(2,1)$ by elements of the subgroup $SO(2).$

The statements 1 and 4 are equivalent to the fact that $d$ is an invariant sub-Riemannian metric on  weakly symmetric
space $(SO_0(2,1)\times SO(2))/SO(2).$ The notion of  weakly symmetric space was introduced by A.~Selberg in his paper
\cite{Selb}; $(SL(2)\times SO(2))/SO(2)$ is unique weakly symmetric nonsymmetric space considered by him in the paper
\cite{Selb}.

\begin{theorem}
\label{theor1}
On $(SO_0(2,1),d),$ any geodesic $\gamma=\gamma(t)=\gamma(\beta,\phi;t),$ $t\in \mathbb{R},$ $\gamma(0)=e$, parametrized by the arclength is the product of two 1--parameter subgroups:
\begin{equation}
\label{sol}
\gamma(t)=\exp(t(\cos\phi\cdot a + \sin \phi\cdot b -\beta c)) \exp(t\beta c),
\end{equation}
where $\phi$ and $\beta$ are arbitrary constants.
\end{theorem}

The next corollary follows from here and the above properties of the metric $d.$

\begin{corollary}
\label{go}
The space $(SO_0(2,1),d)$ is geodesic orbit, i.e.,
every (full) geodesic in $(SO_0(2,1),d)$ is an orbit of  some  1--parameter subgroup of isometries of the space
$(SO_0(2,1),d)$.
\end{corollary}

\begin{remark}
\label{rem2}
To change a sign of $\beta$ in (\ref{sol}) is the same as to change simultaneously a sign of $t$
and to change an angle $\phi$ by the angle $\phi\pm \pi.$
\end{remark}

\begin{theorem}
\label{theor2}
Put
\begin{equation}
\label{mn1}
m=t,\quad n=\frac{t^2}{2}\quad\text{if}\quad\mid\beta\mid=1,
\end{equation}
\begin{equation}
\label{mn2}
m=\frac{\Sh{(t\sqrt{1-\beta^2})}}{\sqrt{1-\beta^2}},\quad
n=\frac{\Ch{(t\sqrt{1-\beta^2})-1}}{1-\beta^2}\quad\text{if}\quad\mid\beta\mid<1,
\end{equation}
\begin{equation}
\label{mn3}
m=\frac{\sin{(t\sqrt{\beta^2-1})}}{\sqrt{\beta^2-1}},\quad
n=\frac{1-\cos{(t\sqrt{\beta^2-1})}}{\beta^2-1}\quad\text{if}\quad\mid\beta\mid>1.
\end{equation}
Then the geodesic of the left-invariant sub-Riemannian metric $d$ on the Lie group $SO_0(2,1)$ (see \,Theorem~\ref{theor1}) is equal to  $\gamma(t)=(\gamma_1(t),\gamma_2(t),\gamma_3(t)),$ where the columns $\gamma_j(t),$
$j=1,2,3,$ are given by formulas
\begin{equation}
\label{geod1}
\gamma_1(t)={\left(\begin{array}{c}
1+n \\
m\cos\phi+\beta n\sin\phi \\
m\sin\phi-\beta n\cos\phi
\end{array}\right),}
\end{equation}
\begin{equation}
\label{geod2}
\gamma_2(t)={\left(\begin{array}{c}
m\cos(\beta t-\phi)+\beta n\sin(\beta t-\phi)\\
n\cos(\beta t-\phi)\cos\phi+\beta m\sin\beta t+(1-\beta^2n)\cos{\beta t}\\
n\cos(\beta t-\phi)\sin\phi-\beta m\cos\beta t+(1-\beta^2n)\sin\beta t
\end{array}\right),}
\end{equation}
\begin{equation}
\label{geod3}
\gamma_3(t)={\left(\begin{array}{c}
\beta n\cos(\beta t-\phi) -m\sin(\beta t-\phi)\\
-n\sin(\beta t-\phi)\cos\phi+\beta m\cos\beta t - (1-\beta^2n)\sin\beta t \\
-n\sin(\beta t-\phi)\sin\phi+\beta m\sin\beta t+(1-\beta^2n)\cos{\beta t}
\end{array}\right).}
\end{equation}
\end{theorem}

In consequence of (\ref{geod1}), (\ref{geod2}), (\ref{geod3}), we have
\begin{equation}
\label{n}
n=c_{11}-1;
\end{equation}
\begin{equation}
\label{sistem1}
\left\{\begin{array}{rl}
c_{22}+c_{33}=2\beta m\sin{\beta t}+(2+n-2\beta^2n)\cos{\beta t}, \\
c_{32}-c_{23}=(2+n-2\beta^2n)\sin{\beta t}-2\beta m\cos{\beta t};
\end{array}\right.
\end{equation}
\begin{equation}
\label{sistem2}
c_{22}-c_{33}=n\cos (\beta t-2\phi),\quad
c_{23}+c_{32}=-n\sin (\beta t-2\phi);
\end{equation}
\begin{equation}
\label{system3}
c^2_{11}-1=c^2_{21}+c^2_{31}=c^2_{12}+c^2_{13}=m^2+\beta^2 n^2.
\end{equation}
It follows from (\ref{n}), (\ref{sistem2}) that
\begin{equation}
\label{nsq}
n^2=n^2(c)=(c_{11}-1)^2=(c_{22}-c_{33})^2+(c_{23}+c_{32})^2.
\end{equation}

\begin{proposition}
\label{so2}
$C\in SO(2)$ $\Longleftrightarrow$ $C\in SO_0(2,1),$  $n(C)=0$.
\end{proposition}

\begin{proposition}
\label{beta0}
If $\beta=0$ then every segment  $\gamma(t)=\gamma(0,\phi;t)$, $0\leq t\leq t_1$, is a shortest arc.
\end{proposition}

\begin{proposition}
\label{T}
For every geodesic $\gamma(\beta,\phi;t),$ $t\in \mathbb{R},$ where
$\beta\neq 0,$ there exists a finite number $T>0$ such that  $\gamma(\beta,\phi;t),$ $0\leq t\leq T,$ is a  noncontinuable shortest arc.
\end{proposition}

It follows from the properties of the metric $d,$ formula (\ref{sol}) and Corollary \ref{go} that
$T$ does not depend on $\phi,$ i.e., $T=T(\beta)$ and $T(-\beta)=T(\beta)$.
Therefore $T=T(\mid\beta\mid).$

\section{Submanifold $\Sim$ in $SO_0(2,1)$}

\begin{proposition}
\label{symm}
Let $C\in SO_0(2,1)- SO(2)$ be a symmetric matrix. Then
$$c_{22}c_{33}-c_{23}c_{32}=c_{11}> 1; \quad c_{22}+c_{33}=1+c_{11}>2; \quad \min(c_{22},c_{33})\geq 1.$$
Additionally, in the last inequality we have the equality if and only if $c_{23}=c_{32}=0;$
in this case, $\max(c_{22},c_{33})=c_{11}.$ The condition $c_{23}=c_{32}=0$ is also equivalent to the condition
$c_{12}c_{13}=c_{21}c_{31}=0.$ Moreover, the eigenvalues of $(2\times 2)$--matrix, bordered by the first row and the first column of the matrix $C$,  are equal to $1$ and $c_{11}.$
\end{proposition}

\begin{proof}
It follows from symmetry of the matrix $C$ and the equality (\ref{cps}) that
\begin{equation}
\label{inv}
C^{-1}=\left(\begin{array}{ccc}
c_{11} & -c_{12} & -c_{13} \\
-c_{12} & c_{22} & c_{23} \\
-c_{13} & c_{23} & c_{33}
\end{array}\right).
\end{equation}
Using (\ref{inv}) and the general inversion rule for matrices, we obtain the next equalities
\begin{equation}
\label{c11}
c_{11}=c_{22}c_{33}-c^2_{23};
\end{equation}
\begin{equation}
\label{c22}
c_{22}=c_{11}c_{33}-c^2_{13},\quad c_{33}=c_{11}c_{22}-c^2_{12};
\end{equation}
\begin{equation}
\label{c12}
c_{12}=c_{12}c_{33}-c_{13}c_{23},\quad -c_{13}=c_{12}c_{23}-c_{13}c_{22},\quad -c_{23}=c_{11}c_{23}-c_{12}c_{13}.
\end{equation}

The equality (\ref{c11}) gives us the first statement of Proposition \ref{symm}.

The summation of the equalities in (\ref{c22}) together with (\ref{inv}) and the identity $CC^{-1}=e$ gives us
$$c_{22}+c_{33}=c_{11}(c_{22}+c_{33})-(c^2_{12}+c^2_{13})=c_{11}(c_{22}+c_{33})-(c^2_{11}-1)$$
and we get at once the second statement of  Proposition \ref{symm}.

The equalities
$$c^2_{12}(c_{33}-1)=c^2_{13}(c_{22}-1)=c^2_{23}(1+c_{11})=c_{12}c_{13}c_{23}$$
follows from (\ref{c12}).
The last equality implies that $c_{12}c_{13}c_{23}\geq 0$ (and in this inequality we have the equality
if and only if $c_{23}=0$) and $c_{22}\geq 1,$ $c_{33}\geq 1,$ at the same time
we have an equality in at least one inequality  (we can't have two equalities, because
$c_{22}+c_{33}=1+c_{11}> 2$) if and only if $c_{23}=0.$ It is remain to prove the last statement.

Let $\lambda_1\leq \lambda_2$ be eigenvalues of the above-mentioned matrix. Then, in consequence of the first equality and the second equality of
Proposition \ref{symm}, we get
$$\lambda_1\cdot\lambda_2=c_{11},\quad \lambda_1+\lambda_2=1+c_{11}.$$
Obviously, the numbers $\lambda_1=1,$ $\lambda_2=c_{11}$ are solutions of this  combined equations.
\end{proof}

\begin{proposition}
\label{ysim}
If $C\in SO_0(2,1)-SO(2)$ and $c_{12}=c_{21},$ $c_{13}=c_{31}$, then $C$ is a symmetric matrix.
\end{proposition}

\begin{proof}
It follows from the conditions on the matrix $C$ and the equality (\ref{cps}) that
\begin{equation}
\label{inv1}
C^{-1}=\left(\begin{array}{ccc}
c_{11} & -c_{12} & -c_{13} \\
-c_{12} & c_{22} & c_{32} \\
-c_{13} & c_{23} & c_{33}
\end{array}\right).
\end{equation}
Using (\ref{inv1}) and the general inversion rule for matrices, we obtain the equalities
$$-c_{23}=c_{11}c_{32}-c_{12}c_{13},\quad -c_{32}=c_{11}c_{23}-c_{12}c_{13}.$$
Subtracting the second equality from the first one, we get $c_{32}-c_{23}=c_{11}(c_{32}-c_{23})$.
Since $C\in SO_0(2,1)-SO(2)$, we have $c_{11}>1$. Therefore $c_{32}=c_{23}$.
\end{proof}

Hereafter, $\Sim$ denotes the set of all symmetric matrices in $SO_0(2,1)-\{1\otimes (-E_2)\}$.

\begin{proposition}
\label{simme}
$$\Sim^{-1}= \Sim;\quad k(\Sim)k^{-1}=\Sim,\quad k\in SO(2).$$
\end{proposition}

\begin{proof}
The first equality is a consequence of (\ref{cps}). Let $s\in \Sim,$ $k\in SO(2).$ Then
$$ksk^{-1}=ksk^T,\quad (ksk^{-1})^T=ks^Tk^T=ksk^{-1}\quad\Rightarrow\,\, ksk^{-1}\in\Sim.$$
\end{proof}

\begin{proposition}
\label{repr}
Every element $C\in SO_0(2,1)$ has a unique representation  in the form $C=s_1k_1$ or $g=k_2s_2,$
where $k_i\in SO(2),$ $s_i\in \Sim;$ $i=1,2.$ Additionally, $k_1=k_2.$
\end{proposition}

\begin{proof}
We shall look for $k_1$ in the form
$$k_1=1\otimes \left(\begin{array}{cc}
\cos \eta & -\sin\eta \\
\sin \eta &  \cos\eta
\end{array}\right).$$
Then one can easily see that the equality $C=s_1k_1,$ $s_1\in \Sim,$ is equivalent, on the ground of Proposition \ref{ysim}, to the
matrix equality
\begin{equation}
\label{phit}
\left(\begin{array}{c}
c_{21}\\
c_{31}
\end{array}\right)=\left(\begin{array}{cc}
\cos \eta & -\sin\eta \\
\sin \eta &  \cos\eta
\end{array}\right)\left(\begin{array}{c}
c_{12}\\
c_{13}
\end{array}\right).
\end{equation}
The vectors $(c_{21},c_{31}),$ $(c_{12},c_{13})$ have the same length $\sqrt{c_{11}^2-1}> 0.$ Therefore, there exists exactly one element
$k_1$ with the required property.

Then in consequence of Proposition \ref{simme},
$$C=s_1k_1=k_1(k^{-1}_1s_1k_1):=k_1s_2,\quad s_2\in \Sim.$$
If there exists another representation $C=k_2s_2',$ then $C=(k_2s_2'k_2^{-1})k_2=s_1k_1$ and $k_2=k_1$ by the uniqueness of the representation $C=s_1k_1;$ then $s_2'=s_2.$
\end{proof}

\begin{proposition}
\label{arb}
For any matrix $C\in SO_0(2,1)$, we have
\begin{equation}
\label{ident}
c_{22}c_{33}-c_{23}c_{32}=c_{11},\quad (c_{22}+c_{33})^2+(c_{32}-c_{23})^2=(1+c_{11})^2.
\end{equation}
\end{proposition}

\begin{proof}
The first equality in (\ref{ident}) follows from (\ref{cps}) and the general inversion rule for matrices. Every matrix $C\in SO_0(2,1)$ can be presented in the form (\ref{geod1}), (\ref{geod2}), (\ref{geod3}). Then by the first equality in (\ref{ident}) and (\ref{nsq}), we have the second equality.
\end{proof}

In the next proposition, we compute the matrix  $k_1$ from Proposition \ref{repr}.

\begin{proposition}
\label{k1}
$$\cos\eta=\frac{c_{22}+c_{33}}{1+c_{11}},\quad \sin\eta=\frac{c_{32}-c_{23}}{1+c_{11}}.$$
\end{proposition}

\begin{proof}
The $(2\times 2)$--matrix, bordered by the first row and the first column of the symmetric matrix
 $s_1$ from Proposition \ref{repr}, is equal to
$$
\left(\begin{array}{cc}
c_{22} & c_{23} \\
c_{32} &  c_{33}
\end{array}\right)\left(\begin{array}{cc}
\cos \eta & \sin\eta \\
-\sin \eta &  \cos\eta
\end{array}\right)=\left(\begin{array}{cc}
c_{22}\cos\eta - c_{23}\sin\eta  & c_{22}\sin\eta +c_{23}\cos\eta  \\
c_{32}\cos\eta - c_{33}\sin\eta  & c_{32}\sin\eta +c_{33}\cos\eta
\end{array}\right).$$
The first column of the matrix $s_1$ coincides with the first column of the matrix $C.$
By the second equality in Proposition \ref{symm}, applied to the matrix $s_1$, we get
\begin{equation}
\label{trace}
(c_{22}+c_{33})\cos\eta +(c_{32}-c_{23})\sin\eta = 1+c_{11}.
\end{equation}
Now Proposition \ref{k1} follows from (\ref{trace}), the second equality in (\ref{ident}), and the equality's case in the
Cauchy--Bunyakowsky--Schwarz inequality.
\end{proof}

\begin{proposition}
\begin{equation}
\label{init}
\Sigma:=\{g\in SO_0(2,1): g=kg^{-1}k^{-1},\quad k\in SO(2)\}=\Sim\cup (\Sim\cdot (1\otimes (-E_2))).
\end{equation}
\end{proposition}

\begin{proof}
Let $g=kg^{-1}k^{-1}$, $k\in SO(2)$ and $g=s_1k_1$ be unique representation from Proposition \ref{repr}.
Then in consequence of Propositions \ref{simme} и \ref{repr},
$$s_1k_1=kk_1^{-1}s_1^{-1}k^{-1}=k_1^{-1}(ks_1^{-1}k^{-1}):=k_1^{-1}s_2,\quad k_1^{-1}=k_1.$$
Then $k_1=e$ or $k_1=1\otimes (-E_2)$ and $g\in \Sim\cup (\Sim\cdot (1\otimes (-E_2))).$ Thus
the left set from formula (\ref{init}) is a part of the right set; clearly, the opposite inclusion is true.
\end{proof}

\begin{proposition}
\label{op}
The matrix $C\in SO_0(2,1)-SO(2)$ is symmetric if and only if
$C = \gamma(0,\phi; t),$ where
\begin{equation}
\label{data1}
\ch t = c_{11}>1,\quad \cos\phi=\frac{c_{12}}{\sqrt{c^2_{11}-1}},\quad
\sin \phi=\frac{c_{13}}{\sqrt{c^2_{11}-1}}.
\end{equation}
\end{proposition}

\begin{proof}
The sufficiency follows from (\ref{geod1}), (\ref{geod2}), (\ref{geod3}). Let us prove the necessity.
Let $C\in SO_0(2,1)-SO(2)$ be a symmetric matrix. It is clear that the conditions $t>0$
and (\ref{data1}) uniquely define $\phi$ and $\gamma(0,\phi; t).$ Let us show that
$C= \gamma(0,\phi; t).$ The elements in the first column and the first row of these two matrices coincide.
All two-dimensional vectors $(c_{21},c_{31}),$ $(c_{22},c_{32}=c_{23}),$ $(c_{23},c_{33})$ aren't null vector, their scalar products (and scalar squares) and also the first vector are defined by elements of the first row and the first column of the matrix  $C,$ because $C\in P(2,1),$ moreover, the second coordinate of the second vector is equal to the first coordinate of the third vector. One can easily see that this fact and the statements of Proposition \ref{symm} guarantee that the second and the third vectors are uniquely defined by elements of the first row and the first column of the matrix $C.$ Since the same statements are valid for the matrix $\gamma(0,\phi;t),$ then these matrices coincide.
\end{proof}

It follows directly from Propositions \ref{op} and \ref{beta0} the next

\begin{proposition}
\label{Sim}
The set $\Sim$ of all symmetric matrices from $SO_0(2,1)-SO(2)$ is a union of all geodesics--shortest arcs
$\gamma(t)=\gamma(0,\phi;t)$, $0\leq t\leq t_1,$ joining the unit $e$ with elements from $\Sim$.
\end{proposition}

\section{Cut loci and conjugate sets in $(SO_0(2,1),d)$}

Unlike the Riemannian manifolds, the exponential map $\Exp_x,$ $x\in M,$
for a sub-Riemannian manifold $(M,d)$ with no abnormal geodesic (as in the case of $(SO_0(2,1),d)$) are defined not on $TM$ and $T_xM$ but on $D(x)\times \Ann(D(x))$, where $D$ is the \textit{distribution} on $M$ involved in the definition of $d,$ and
$$\Ann(D(x))=\{\psi\in T^{\ast}_xM: \langle\psi,D(x)\rangle=0\},$$
see \cite{VG}.
Otherwise, the cut loci and conjugate sets for such sub-Riemannian manifolds are defined in the same way as for Riemannian ones \cite{GKM}.

\begin{definition}
\label{cutl}
Cut locus $C(x)$ (respectively, (the first) conjugate set $S(x)$ ($S_1(x)$)) for a point $x$ in a sub-Riemannian manifold $M$ (without abnormal geodesics) is the set of ends of all shortest arcs starting at the point $x$ and noncontinuable beyond its ends
(respectively, the image of the set of (the first) critical points (along geodesics with the origin in $x$) of the map $\Exp_x$
with respect to $\Exp_x$).
\end{definition}

\begin{proposition}
\label{se}
$$S(e)=(S_1(e)=SO(2)-\{e\}) \quad\cup $$
$$\left\{\gamma(\beta,\phi;t)\mid \tg\left(\frac{t\sqrt{\beta^2-1}}{2}\right)=\frac{t\sqrt{\beta^2-1}}{2},\,\,\beta^2>1,\,\, t\neq 0\right\};$$
$$C\cap S(e)=C\cap S_1(e)=SO(2)-\{e\}.$$
\end{proposition}

\begin{proof}
The Lie group $SL(2)/\{\pm e\}$ is isomorphic to the Lie group $SO_0(2,1)$ (see, for example, \cite{Ber}). By Theorem 1 from \cite{Ber}, there exists
a locally isomorphic epimorphism of the Lie groups
$$L:SL(2) \rightarrow SL(2)/\{\pm e\}\cong SO_0(2,1)$$
such that, in terms of this paper and paper \cite{BerZub},
$$dL(e)(p_1)=a,\quad dL(e)(p_2)=b,\quad dL(e)(k)=c.$$
At the same time, $(p_1,p_2)$ is an orthonormal basis of the vector subspace of the Lie algebra
$\mathfrak{sl}(2)$ of the Lie group $SL(2),$ defining the left-invariant sub-Riemannian metric
$\delta$ on the Lie group $SL(2).$ This implies that the map
$L:(SL(2),\delta)\rightarrow (SO_0(2,1),d)$ is
a submetry \cite{BG} and  local isometry. Consequently, $L$ maps the geodesics of the space
$(SL(2),\delta)$ (with a parameter $\beta$) to the geodesics of the space $(SO_0(2,1),d)$ (with the same parameter).
Therefore, the  critical values of the map $\Exp_e$ for $(SL(2),\delta)$ and $(SO_0(2,1),d)$ coincide. Besides,
$L(SO(2)\subset SL(2))=SO(2)\subset SO_0(2,1).$ The proposition \ref{se} with the same statement was
proved as Proposition 11 in the paper \cite{BerZub} for the space $(SL(2),\delta).$ As a result, it holds for the space
$(SO_0(2,1),d).$
\end{proof}

The statements from the following proposition were actually proved in \cite{Ber}.

\begin{proposition}
\label{kk1}
A matrix $C\in C(e)$ if and only if $C\in S_1(e)=SO(2)-\{e\}$ or there exist
$\phi_1$, $\phi_2\in\mathbb{R}$, $0< \beta < \frac{3}{\sqrt{5}}$ such that
\begin{equation}
\label{equat}
C=\gamma(\beta,\phi_1;T)=\gamma(\beta,\phi_2;-T),
\end{equation}
where $\gamma(t)$ is defined by (\ref{geod1}), (\ref{geod2}), (\ref{geod3}), and
\begin{equation}
\label{zT}
T=\min\{t>0\mid \gamma(\beta,\phi_1;t)=\gamma(\beta,\phi_2;-t)\}.
\end{equation}
\end{proposition}

The main result of this section constitutes

\begin{theorem}
\label{cutloc}
For every element $g\in (SO_0(2,1),d)$, $C(g)=gC(e)$ and $S(g)=gS(e)$. Moreover,
\begin{equation}
\label{cutloc0}
C(e)=K(e)\cup S_1(e),
\end{equation}
where
\begin{equation}
\label{cutloc1}
K(e)=\left\{C \in SO_0(2,1)\,\mid\,c_{21}=-c_{12},\,\,c_{31}=-c_{13},\,\,c_{23}=c_{32},\,\,c_{22}+c_{33}<0\right\},
\end{equation}
\begin{equation}
\label{congloc}
S_1(e)=SO(2)-\{e\}.
\end{equation}
\end{theorem}

\begin{proof}
In consequence of Propositions \ref{se}, \ref{kk1}, it is enough to prove that
\begin{equation}
\label{eq}
C(e)-SO(2)=K(e)-SO(2).
\end{equation}
Note first of all that
\begin{equation}
\label{K}
K(e)-SO(2)= (\Sim-\{e\})\cdot (1\otimes (-E_2))= (1\otimes (-E_2))\cdot (\Sim-\{e\})
\end{equation}
on the ground of (\ref{cutloc1}), (\ref{phit}), and Proposition \ref{repr}.

It follows from (\ref{sol}) that
$$\gamma(\beta,\phi_1;-t)=\exp(t\beta c)\gamma(\beta,\phi_1,t)^{-1}\exp(-t\beta c).$$
Therefore one can rewrite the equality $\gamma(\beta,\phi_1;t)=\gamma(\beta,\phi_2;-t)$ in the form
\begin{equation}
\label{ini}
\gamma(\beta,\phi_1;t)=k\gamma(\beta,\phi_1,t)^{-1}k^{-1},\quad k\in SO(2).
\end{equation}
Then in consequence of (\ref{init}), we get
$$\gamma(\beta,\phi_1;t)\in \Sigma = \Sim\cup (\Sim\cdot (1\otimes (-E_2))).$$
It follows from here, Proposition \ref{Sim} and (\ref{K}) that $C\in \Sim\cdot (1\otimes (-E_2))=K(e)$
for any matrix $C$  from (\ref{equat}), if $0<\beta < \frac{3}{\sqrt{5}}.$
Thus, in consequence of Proposition \ref{kk1}, we have $C(e)-SO(2)\subset K(e)-SO(2).$

Let $C\in K(e)-SO(2).$ Then in consequence of (\ref{K}) and (\ref{init}),
$$C\in (\Sim-\{e\})\cdot (1\otimes (-E_2))\subset \Sigma.$$
By the Cohn--Vossen theorem \cite{CF}, there exists a shortest arc
$\gamma(\beta,\phi_1;t),$ $0\leq t \leq t_1,$ joining $e$ and $C.$ By Proposition \ref{Sim}, we have
$0< \mid\beta\mid < \frac{3}{\sqrt{5}}$. Moreover, if $t=t_1$ then holds the relation (\ref{ini}) that can be rewritten
in form $\gamma(\beta,\phi_1;t_1)=\gamma(\beta,\phi_2;-t_1)$ by (\ref{sol}).
On the ground of Corollary \ref{rem2}, we can assume that $\beta>0.$ Thus $t_1=T$ (see (\ref{zT})) because
we consider a shortest arc. Proposition \ref{kk1} implies that $C\in C(e)-SO(2).$
\end{proof}

\section{Sub-Riemannian distance}

The next proposition reduces the distance search to the search of the $\mid\beta\mid,$ \textit{the geodesic curvature} of the projection $p(\gamma(t)),$ $t\in \mathbb{R},$ of the geodesic (\ref{sol}) onto $L^2$ \cite{Ber}.

\begin{proposition}
\label{dist}
If $C\in SO_0(2,1)-\{e\}$ then
\begin{equation}
\label{distance}
\mid\beta\mid d(C,e)=\Arccos\left(\frac{c_{22}+c_{33}}{1+c_{11}}\right)+
2\arccos\left(\frac{m}{\sqrt{c^2_{11}-1}}\right):=\mid\eta\mid + 2\arccos\mu,
\end{equation}
\begin{equation}
\label{md}
\mid\mu\mid=\frac{\mid m\mid}{\sqrt{c^2_{11}-1}}:=\sqrt{1-\beta^2\frac{c_{11}-1}{1+c_{11}}},
\end{equation}
where $\beta$ and $m$ are given from the representation of the matrix $C$ in the form (\ref{geod1}), (\ref{geod2}), (\ref{geod3}), and
$t=d(C,e)$. Thus the following statements hold.

I. If $C\in SO(2)$ then $\mu=-1.$

If $C\notin SO(2)$ then

II. $\mu=0$ and $\beta^2=(1+c_{11})/(c_{11}-1)$ for $\mid \eta \mid=\pi(\sqrt{(1+c_{11})/2}-1):=\theta.$

III. $\mu<0$ for $\mid\eta\mid> \theta.$

IV. $\mu>0$ for $\mid\eta\mid< \theta.$
\end{proposition}

\begin{proof}
Let $\gamma(t):=\gamma(\beta,\phi;t),$ $0\leq t\leq t_1,$ be a line segment of the geodesic in $(SO_0(2,1),d)$ of the form  (\ref{sol}), which is a shortest arc. Then its projection
\begin{equation}
\label{pro}
x(t)=p(\gamma(t)),\quad 0\leq t\leq t_1,
\end{equation}
onto $L^2$ is either the circle $S^1$ or an unclosed curve which has no self-intersection \cite{Ber}. On the ground of the statement 3) from Introduction, the first case is characterized by the condition $\gamma(t_1)\in SO(2)$. Let $S(t_1)$ be the area of a region in $L^2$ bounded by $S^1$ or the curvilinear digon $P,$ consisting of the curve (\ref{pro}) and the shortest arc $[x(0),x(t_1)]$ with the length $r=r(t_1)$ and the interior angle $\psi=\psi(t_1)$ in $L^2$. Thus
$S(t_1)\leq \pi$ by paper \cite{Ber}.

At first consider the case $C=\gamma(t_1)\notin SO(2).$
In \cite{Ber} (see formula (44)), are found the following equations for the digon $P$:
$$S(t_1)=\mid\beta\mid t_1 - 2\psi,\quad r=\Arch(1+n),\quad \cos\psi = \frac{m}{\sqrt{n(n+2)}}.$$
By the statement 3) from Introduction, Proposition \ref{k1} and (\ref{n}), (\ref{system3}), these equations can be rewritten in the form
\begin{equation}
\label{src}
\mid\eta\mid=\mid\beta\mid t_1 -2\psi,\quad r=\Arch c_{11},\quad \cos\psi = \sgn(m)\mid\mu\mid,
\end{equation}
\begin{equation}
\label{src1}
\mid\eta\mid=\arccos\left(\frac{c_{22}+c_{33}}{1+c_{11}}\right),\quad \mid\mu\mid=\sqrt{1-\beta^2\frac{c_{11}-1}{1+c_{11}}}.
\end{equation}

Obviously, $\mu=0\,\Longleftrightarrow\,\psi=\pi/2$. This is equivalent to the fact that $P$ bounds a semidisc of the radius $r/2$ in $L^2.$ In this case, in consequence of (46) in \cite{Ber} and (\ref{src}), we have $\mid\beta\mid=\sqrt{\frac{1+c_{11}}{c_{11}-1}}$ and
$$\mid\eta\mid= S(t_1)=\pi\left(\ch\frac{r}{2}-1\right)=\pi\left(\sqrt{\frac{1+c_{11}}{2}}-1\right):=\theta.$$
It is clear from geometric considerations that
$$\mu >0\,\,\Longleftrightarrow\,\,\mid\eta\mid < \theta,\quad \mu < 0\,\,\Longleftrightarrow\,\,\mid\eta\mid > \theta.$$

Let $C=\gamma(t_1)\in SO(2)$ and $t_n\nearrow t_1$ for
$n\rightarrow \infty,$ $n\geq 2,$ $\psi_n=\psi(t_n),$ $\mu_n=\mu(t_n).$ Then it is clear that
$$\gamma(t_n)\rightarrow C,\quad \psi_n\nearrow \pi,\quad \mu_n\searrow \mu=-1,\quad d(e,\gamma(t_n))=t_n\nearrow t_1=d(e,C).$$
Therefore the equality (\ref{distance})  holds for $\mu=-1.$
\end{proof}

\begin{theorem}
\label{distan}
0. If $C\in \Sim$ then
$$d(e,C)=\Arch c_{11}.$$

If $C\notin \Sim$ then hold the following statements.

I. If $C\in SO(2)$ then
\begin{equation}
\label{dpp}
\mid\beta\mid\geq \frac{3}{\sqrt{5}},\quad d(e,C)=\frac{2\pi}{\sqrt{\beta^2-1}}
\end{equation}
and $\mid\beta\mid$ is a unique number defined by equality
\begin{equation}
\label{mbp}
\frac{2\mid\beta\mid\pi}{\sqrt{\beta^2-1}}= \Arccos\left(\frac{c_{22}+c_{33}}{1+c_{11}}\right)+ 2\pi.
\end{equation}

II. If $\mid\eta\mid =\pi(\sqrt{(1+c_{11})/2}-1):=\theta$ then $\mid\beta\mid=\sqrt{\frac{1+c_{11}}{c_{11}-1}},$\,\,\,$d(e,C)=\pi \sqrt{\frac{c_{11}-1}{2}}.$ At the same time, $\mid\beta\mid=\frac{2}{\sqrt{3}}$ only if $\mid\eta\mid=\pi;$
otherwise we have $\mid\beta\mid> \frac{2}{\sqrt{3}}.$

III. If $\mid\eta\mid > \theta$ then\,\,\,$\frac{2}{\sqrt{3}}< \mid\beta\mid < \sqrt{\frac{1+c_{11}}{c_{11}-1}}.$ If
$\mid\eta\mid=\pi$ then $\mid\beta\mid\leq \frac{3}{\sqrt{5}}.$ Generally,
\begin{equation}
\label{d}
d(e,C)=(2\pi-\gamma)/\sqrt{\beta^2-1},\quad\mbox{where}
\end{equation}
\begin{equation}
\label{et}
\mid\eta\mid=(2\pi-\gamma)\frac{\mid\beta\mid}{\sqrt{\beta^2-1}}-2\arccos(-\mid\mu\mid),\quad 0\leq \gamma<\pi;
\end{equation}
\begin{equation}
\label{sin}
\cos\gamma=\beta^2-c_{11}(\beta^2-1).
\end{equation}

IV. If $\mid\eta\mid < \theta$ then the following statements are valid.

a) If $\mid\eta\mid= \sqrt{2(c_{11}-1)}-2\arccos\sqrt{\frac{2}{1+c_{11}}}:=\alpha$ then
$$\mid\beta\mid=1,\quad d(e,C)=\sqrt{2(c_{11}-1)}.$$
b) If $\mid\eta\mid < \alpha$ then $0< \mid\beta\mid< 1,$ and $\mid\beta\mid$ is an implicit solution of equation
$$\mid\eta\mid= 2\left(\frac{\mid\beta\mid}{\sqrt{1-\beta^2}}\Arch\left(\sqrt{\frac{1+c_{11}}{2}}\mid\mu\mid\right)
-\arccos\mid\mu\mid\right),$$
$$d(e,C)=\frac{2}{\sqrt{1-\beta^2}}\Arch\left(\sqrt{\frac{1+c_{11}}{2}}\mid\mu\mid\right).$$
c) If $\mid\eta\mid > \alpha$ then $1< \mid\beta\mid< \sqrt{\frac{1+c_{11}}{c_{11}-1}}.$ If $\mid\eta\mid=\pi$
then $\mid\beta\mid < \frac{2}{\sqrt{3}}.$ Generally, we have (\ref{sin}) and
\begin{equation}
\label{d1}
d(e,C)=\gamma/\sqrt{\beta^2-1},\quad\mbox{where}
\end{equation}
\begin{equation}
\label{et1}
\mid\eta\mid=\gamma\frac{\mid\beta\mid}{\sqrt{\beta^2-1}}-2\arccos\mid\mu\mid,\quad 0< \gamma <\pi.
\end{equation}

V. The above-mentioned conditions uniquely define $d(e,C).$
\end{theorem}

\begin{proof}
0. This statement is a consequence of Proposition \ref{op}.

I. Relations (\ref{dpp}) were proved in \cite{Ber}. The equality (\ref{mbp}) follows from Proposition \ref{dist}. The uniqueness $\mid\beta\mid$ in this equality follows from inequalities
$$\mid\beta\mid\geq \frac{3}{\sqrt{5}}\Longleftrightarrow 2\pi< \frac{2\mid\beta\mid\pi}{\sqrt{\beta^2-1}}\leq 3\pi.$$

II. The first statement is a consequence of Proposition \ref{dist}. If $\mid\eta\mid=\pi$ then it follows from the first statement that
$$c_{11}=7,\quad \mid\beta\mid=\frac{2}{\sqrt{3}},\quad d(e,C)=\pi\sqrt{3}.$$

Further we shall use the notation from the proof of Proposition \ref{dist}.

If $\mid\eta\mid=S(P)<\pi$ then $\mid\beta\mid>\frac{2}{\sqrt{3}},$ since $P$ is a semidisc in the case II, and the area of a disc in $L^2$ is an increasing function of the geodesic curvature $\mid\beta\mid$ of its bounding circle. More precisely, the area of a disc with radius $R,$ that is equal to $2\pi(\ch R-1),$ is an increasing function of
$R,$ while $\mid\beta\mid=\cth R$ is a decreasing function of $R$ by the first formula in (\ref{cth}).

III. In this case, in consequence of Proposition \ref{dist}, $P$ is a disk or the bigger segment of a disk with radius
$R> r/2$ in $L^2,$ with some center $O,$ and a base with the length $r=\Arch c_{11}.$ The first case, $C\in SO(2)$, was considered in I. In the second case,
$$\frac{\pi}{2}<\psi= \arccos(-\mid\mu\mid) < \pi,\quad 1< \mid\beta\mid < \sqrt{\frac{1+c_{11}}{c_{11}-1}}.$$
Let $\gamma$ be an angle in the triangle $x(0)Ox(t_1)$ at the vertex $O,$  $\gamma_1$ be an
angle at vertices $x(0),$ $x(t_1).$ Then $\gamma_1=\psi-\pi/2.$
The area $S(t_1)$ is equal to the sum of the area of the sector of the radius $R$ with  central angle $2\pi-\gamma$ and the area of the triangle  $x(0)Ox(t_1);$ by the Gauss--Bonnet theorem, the area of a triangle is equal to its excess taken with the minus sign. Therefore,
$$\mid\eta\mid=S(t_1)=(2\pi-\gamma)(\ch R-1)+(\pi-\gamma-2(\psi-\pi/2))=(2\pi-\gamma)\ch R-2\psi.$$
Then in consequence of (\ref{src}),
$$\mid\beta\mid d(e,C)=\mid\beta\mid t_1=\mid\beta\mid(2\pi-\gamma)\sh R=\mid\eta\mid + 2\psi=(2\pi-\gamma)\ch R.$$
Hence
\begin{equation}
\label{cth}
\mid\beta\mid=\cth R,\quad \sh R = \frac{1}{\sqrt{\beta^2-1}},\quad \ch R = \frac{\mid\beta\mid}{\sqrt{\beta^2-1}}.
\end{equation}
By the (first) cosine theorem in the hyperbolic geometry \cite{Berd} и (\ref{cth}), we have
\begin{equation}
\label{cosh}
\ch r = \ch^2R-\sh^2R\cdot\cos\gamma,\quad \cos\gamma=\frac{\ch^2R-\ch r}{\sh^2R}=\beta^2-c_{11}(\beta^2-1).
\end{equation}
Using the above-mentioned relations, we get the equalities (\ref{d}), (\ref{et}), (\ref{sin}).

All statements in III, except the second statement and the second inequality in the first statement, are proved. In our case, we have $\pi \geq \mid\eta\mid=S(P)> S(P_1),$ where $P_1$ is a semidisk of the same radius as $P,$ contained in $P.$ Then in consequence of the same considerations as in the proof of item II, we have $\mid\beta\mid >\frac{2}{\sqrt{3}}$.

Let $\mid\eta\mid=s(P)=\pi.$ If $P$ is a disk of the radius $R$ then
$$2\pi(\ch R-1)=\pi,\quad \ch R=\frac{3}{2},\quad \mid\beta\mid = \cth R=\frac{3}{2\sqrt{(3/2)^2-1}}=\frac{3}{\sqrt{5}}.$$
If $P$ isn't the disc then $\mid\beta\mid > \frac{3}{\sqrt{5}}$ as a corollary of the same argument as above.

IV. Assume at first that $1< \mid\beta\mid <\sqrt{\frac{1+c_{11}}{c_{11}-1}}.$ In this case, we must compute the area $S$ of a smaller segment of a disk of radius $R>r/2,$ cut off by the line segment $[x(0),x(t_1)],$ and the angles
$\psi_2=\pi-\psi>\pi/2,$ $\gamma_2=\pi/2-\psi<\pi/2$ plays the role of
the angles $\psi,$ $\gamma_1$ for the bigger segment with area $S'$. Then, as in item III,
$$S'=(2\pi-\gamma)\ch R-2\psi_2 = (2\pi-\gamma)\ch R-2(\pi-\psi),$$
$$S=2\pi(\ch R-1)-S'=\gamma\ch R-2\psi.$$
The same relations (\ref{cth}), (\ref{cosh}), and so the equalities  (\ref{sin}), (\ref{d1}), (\ref{et1}) are valid,
then by the sine theorem,
\begin{equation}
\label{sint}
\frac{\sh R}{\sh r}=\frac{\sin\gamma_2}{\sin\gamma}=\frac{\cos\psi}{\sin\gamma}.
\end{equation}
If $\mid\eta\mid=S(P)=\pi$ then $$\pi< \pi(\ch R-1),\quad \ch R > 2,\quad \mid\beta\mid=\cth R < \frac{2}{\sqrt{2^2-1}}=\frac{2}{\sqrt{3}}.$$

a) The case of the horocycle $\mid\beta\mid=1$ is obtained from the case just considered for $R\nearrow \infty$ and fixed
$c_{11}.$ In this limit passage, $\psi$ decreases, therefore $\cos\psi$ increases and in concequence of
(\ref{sint}),
$$\cos\psi>1/\sqrt{1+c_{11}},\quad \gamma\searrow 0,\quad \sin\gamma\searrow 0,\quad\gamma\sim\sin\gamma,$$
$$\mid\eta(1)\mid=\lim_{R\rightarrow\infty}(\gamma\ch R-2\psi)=\lim_{R\rightarrow\infty}\left(\frac{\cos\psi\sqrt{c^2_{11}-1}\ch R}{\sh R}-2\psi\right)=$$
$$\cos\psi(1)\sqrt{c^2_{11}-1}-2\psi(1),$$
$$\cos\psi(1)=\mid\mu(1)\mid=\sqrt{1-\frac{c_{11}-1}{1+c_{11}}}=\sqrt{\frac{2}{1+c_{11}}}.$$
From here follows items a), c) and the first statement of item b).

b) A semigeodesic coordinate system in $L^2$ has the form $ds^2=du^2+\ch^2(u)dv^2$ (see, for example, \cite{Ber}).

Consider a curvilinear quadrangle  $P_1:\,\,0\leq u\leq\sigma,\,\, -l\leq v\leq l.$
Lines $0\leq u\leq \sigma,\,\, v=-l\,\,\mbox{or}\,\,v=l;$  $-l\leq v\leq l,\,\,u=0,$ are straight line segments, and line $u=\sigma$ is an equidistant curve on the distance $\sigma$ from the straight line $u=0.$
The area $S_1$ of the quadrangle $P_1$ is equal to
$$S_1=2l\int_0^{\sigma}\ch udu=2l\sh\sigma.$$
The Gauss-Bonnet theorem (see \cite{Pog} or Theorem 5 in \cite{Ber}) applied to $P_1$ gives
$$\mid\beta\mid 2l\ch\sigma+4\left(\pi-\frac{\pi}{2}\right)=2\pi+2l\sh\sigma.$$
Consequently,
\begin{equation}
\label{thsh}
\mid\beta\mid=\Th\sigma,\quad \sh\sigma=\frac{\Th\sigma}{\sqrt{1-\Th^2\sigma}}=\frac{\mid\beta\mid}{\sqrt{1-\beta^2}},
\quad\ch\sigma=\frac{1}{\sqrt{1-\beta^2}},
\end{equation}
where $\mid\beta\mid=\Th\sigma$ is the geodesic curvature of the equidistant curve $u=\sigma.$

The rectilinear "Saccheri quadrangle"\,$P_2$ with the same vertices as $P_1,$ lies in $P_1.$ Suppose that its upper base has the length $r=\ch c_{11}.$ Then quadrangle $P=\overline{P_1-P_2}$ has the required area
$\mid\eta\mid=S(t_1)=S_1-S_2.$ The $P_2$ has two angles, which are equal to $\pi/2,$ and two angles, that are equal to
$\phi=\pi/2-\psi.$  Then the Gauss-Bonnet theorem applied to $P_2$ gives
$$2\left(\pi-\frac{\pi}{2}\right)+2\left[\pi-\left(\frac{\pi}{2}-\psi\right)\right]=2\pi + S_2\,\,\Longrightarrow\,\,S_2=2\psi,
\,\,\mid\eta\mid=2l\sh\sigma-2\psi.$$

The half of $P_2,$ selected from it by conditions $0\leq v\leq l,$ is a "Lambert quadrangle".
Applying to it Theorem 7.17.1, (ii) from \cite{Berd}, we get
$$\ch l=\ch\left(\frac{r}{2}\right)\sin\phi= \ch\left(\frac{r}{2}\right)\sin\left(\frac{\pi}{2}-\psi\right)=\sqrt{\frac{1+c_{11}}{2}}\cos\psi.$$
Then as a consequence of the third equalities in (\ref{src}) and (\ref{src1}),
$l=\Arch\left(\mid\mu\mid\sqrt{\frac{1+c_{11}}{2}}\right).$
The statements of b) in Theorem follow from here, (\ref{thsh}) and the equalities
$$d(e,C)=t_1=2l\ch\sigma,\quad \mid\eta\mid=2l\sh\sigma-2\psi.$$

V. The statement follows from Propositions \ref{repr}, \ref{k1} and from the fact that for any line segment in $L^2$ with a given length $r=\Arch c_{11}>0$ there exists a unique up to the line segment reflection curve of constant geodesic curvature bounding together with the segment a region with a given area $\mid\eta\mid=S\geq 0.$
\end{proof}

\end{document}